\documentclass[11pt]{amsart}

\usepackage{amsfonts, amstext, amsmath, amsthm, amscd, amssymb}
\usepackage{graphicx, color}
\usepackage{import}

\usepackage{microtype}

\usepackage[hidelinks,pagebackref,pdftex]{hyperref}

\makeatletter
\define@key{href}{font}{#1}
\makeatother
\usepackage{xpatch}
\newcommand\hrefdefaultfont{\ttfamily}
\xpatchcmd\href{\setkeys{href}{#1}}{\setkeys{href}{font=\hrefdefaultfont,#1}}{}{\fail}

\renewcommand*{\backref}[1]{}
\renewcommand*{\backrefalt}[4]{
  \ifcase #1 
  [No citations.]
  \or [#2]
  \else [#2]
  \fi }

\newtheorem{theorem}{Theorem}[section]
\newtheorem{proposition}[theorem]{Proposition}
\newtheorem{lemma}[theorem]{Lemma}
\newtheorem{corollary}[theorem]{Corollary}

\newtheorem*{namedtheorem}{\theoremname}
\newcommand{\theoremname}{testing}

\theoremstyle{definition}
\newtheorem{definition}[theorem]{Definition}

\newcommand{\refthm}[1]{Theorem~\ref{Thm:#1}}
\newcommand{\reflem}[1]{Lemma~\ref{Lem:#1}}
\newcommand{\refprop}[1]{Proposition~\ref{Prop:#1}}
\newcommand{\refcor}[1]{Corollary~\ref{Cor:#1}}

\newcommand{\refsec}[1]{Section~\ref{Sec:#1}}
\newcommand{\reffig}[1]{Figure~\ref{Fig:#1}}

\newcommand{\RR}{\mathbb{R}}  
\newcommand{\bdy}{\partial}   

\title[Lorenz links obtained by twisting]{Hyperbolic and satellite Lorenz links obtained by twisting}

\author{Thiago de Paiva}
\address[]{School of Mathematics, Monash University, Clayton, VIC 3800, Australia }
\email[]{thiago.depaivasouza@monash.edu}

\author{Jessica S. Purcell}
\address[]{School of Mathematics, Monash University, Clayton, VIC 3800, Australia }
\email[]{jessica.purcell@monash.edu}


\begin{document}

\begin{abstract}
A Lorenz link is equivalent to a T-link, which is a positive braid built by concatenating torus braids of increasing size. When each torus braid except the largest is obtained by full twists, then the T-link can be described as the Dehn filling of a parent link. In this paper, we completely classify when such parent links are hyperbolic. This gives a classification of the geometry of T-links obtained by full twists when the amount of twisting is large, although the bound on the number of required twists is not effective. We also present effective results on hyperbolicity for broad families of T-links obtained by twisting, even when the number of twists is small. Finally, we identify families of satellite T-links obtained by half-twists. 
\end{abstract}

\maketitle

\section{Introduction}

Lorenz links are the closed periodic orbits of a system of equations investigated by Lorenz in the 1960s~\cite{Lorenz}. They exhibit interesting dynamics that has led to significant further investigation over the years, in the fields of dynamics, geometry, and topology; see for example~\cite{GhysLeys}. These links can be described as links on an embedded branched surface in $\RR^3$, called the Lorenz template, due to work of Guckenheimer and Williams~\cite{GuckenheimerWilliams}, and Tucker~\cite{Tucker}. Birman and Williams were the first to investigate Lorenz links through the lens of knot theory, in the 1980s~\cite{BirmanWilliams}, and the first to show such links are closed positive braids. Birman and Kofman~\cite{BirmanKofman} showed that Lorenz links are equivalent to T-links, which are positive braids with a particular form; see \refsec{Braids} below. Thus techniques from braid theory can be brought to bear upon Lorenz links via T-links.

We are interested in the complement of these links, and in particular their geometrization. Thurston showed in the 1980s that all knots in the 3-sphere are either torus knots, satellite, or hyperbolic~\cite{thurston:bulletin}, and we refer to this as the knot's geometric type. The geometric type of Lorenz links has been considered since work of Birman and Williams in the 1980s~\cite{BirmanWilliams}. They showed that all torus knots are Lorenz knots, and satellites obtained as certain cables of Lorenz knots are Lorenz knots. The hyperbolic geometry of Lorenz knots has been considered by Gomes, Franco, and Silva~\cite{GomesFrancoSilva:Partial, GomesFrancoSilva:Farey}, who proved hyperbolicity of certain Lorenz links satisfying conditions based on the Lorenz template. Satellite links have received additional attention, by El Rifai~\cite{ElRifai:Satellite}, de~Paiva~\cite{de2022satellite}, and de~Paiva and Purcell~\cite{depaiva-purcell2021satellites}.

In spite of this work, there remains no systematic way of determining whether a Lorenz link is hyperbolic, toroidal, or satellite using its description either on the Lorenz template, or as a closed braid in the form of a T-link. These descriptions uniquely determine a link, and hence uniquely determine its geometric type, so it is natural to ask for a simple description of geometric type based on the Lorenz link description. 

In this paper, we consider the Birman--Kofman description of Lorenz links as T-links. We extend the classification of the geometry of T-links to those that admit full twists. Such links are obtained from a ``parent'' link by Dehn filling. Such parent links can give significant insight into the geometry of their Dehn fillings. The ones considered in this paper are generalizations of links considered by Lee in the simpler twisted torus knot case~\cite[Proposition~5.7]{Lee:Unknotted}. They also appear in work used to obtain upper volume bounds on related links in \cite{CFKNP}. In this paper, we completely classify when these parent links are hyperbolic; this is \refthm{AugHyperbolic} below.
This result leads to new infinite families of hyperbolic T-links, determined only by parameters in a braid describing the link.

\begin{theorem}\label{Thm:MainGeneric}
Fix relatively prime integers $1<q<p$, and integers $1<a_1<\dots<a_n<p$. There exists $B\gg 0$ with the following property. 
Consider the T-link obtained from the $(p,q)$-torus knot by full twisting at least $B$ times in regions with $a_1, a_2, \dots, a_n$ strands, respectively. This T-link is hyperbolic if and only if one of the following holds:
\begin{itemize}
\item all the $a_i < q$, or
\item $a_n=q$ and $n>1$, or
\item there is $a_i>q$ that is not a multiple of $q$.
\end{itemize}
\end{theorem}

The T-links of \refthm{MainGeneric} must be obtained by full twisting. Unfortunately the constant $B$ in the above theorem is not effective, thus we currently do not have a concrete, universal bound on the number of full twists that are required.
However, in \refsec{SomeParameters} we improve this: We present two theorems that guarantee hyperbolicity of T-links with full twists, given only their parameters, where the bounds on numbers of full twists required are explicit, small, and relatively simple. The results are \refthm{2Twists_an_p_RelPrime} and \refthm{SomeParamsHyperbolic}. For example, the following result is a corollary of \refthm{2Twists_an_p_RelPrime}.

\begin{corollary}\label{Cor:SpecialGeneric}
Suppose $p$, $q$, and $a_n$ are relatively prime integers, with $1<q<a_n<p$, and suppose $a_1< \dots< a_n$ are positive integers less than $p$. Let $K$ be the T-link obtained from the $(p,q)$-torus knot by full twisting at least once in regions with $a_1, a_2, \dots, a_{n-1}$ strands and at least twice in the region with $a_n$ strands. Then $K$ is hyperbolic. 
\end{corollary}

This result is quite strong, especially compared with similar results on hyperbolicity when multiple strands are twisted. For example Purcell~\cite[Theorem~3.7]{Purcell:HyperbolicMultiply} guarantees hyperbolicity of certain links obtained by Dehn filling, twisting multiple strands, but six full twists are required in addition to other symmetries for the methods of that paper to apply. Note \refcor{SpecialGeneric} requires only one or two full twists. Similarly, all the T-links that are hyperbolic by Theorems~\ref{Thm:2Twists_an_p_RelPrime} and~\ref{Thm:SomeParamsHyperbolic} require diagrams with only one or two full twists. This is because the proofs of these results have access to tools from braid theory, particularly work of Ito~\cite{Ito}, and do not strictly rely on geometry. 

Note that the description of Lorenz links via T-links does not require full twists, and so it would be nice to weaken the hypotheses of \refthm{MainGeneric} and \refcor{SpecialGeneric} even further to allow more general torus braids in regions with $a_j$ strands. This seems difficult. In this paper, we give more results in the satellite case. 

\begin{theorem}\label{Thm:MainSatellite}
For $q<p$ integers, let $K$ be a T-link obtained from the $(p,q)$-torus link by half-twisting in circles encircling less than $q$ strands, or encircling multiples of $q$ strands. Then $S^3-K$ is satellite. 
\end{theorem}

The precise statement is \refthm{MainHalfTwist}. Note this extends work of~\cite{depaiva-purcell2021satellites} to families of T-links with both full twists and \emph{half twists}, which gives many more families in a very natural way. 

\subsection{Acknowledgements}
This work was partially supported by the Australian Research Council, grant DP210103136. 

\section{Results on braids}\label{Sec:Braids}

This section reviews results on braids that will be used throughout. As usual, let $\sigma_i$ be the standard generator of the braid group, giving a positive crossing between the $i$-th and $(i+1)$-th strands.

For $1<p, q$, define the $(p,q)$-torus braid as:
\[ (\sigma_1\dots \sigma_{p-1})^q \]
Note that within the braid group on $p$ strands, its closure is the torus link $T(p,q)$. When $p,q$ are coprime, this is a torus knot, but we will not always restrict to coprime $p$ and $q$ unless specifically stated.

We will also consider such braids within larger braid groups. When $r<p$, the $(r,s)$ braid within the braid group on $p$ strands is still defined to be $(\sigma_1 \dots \sigma_{r-1})^s$, but now note this has $p-r$ strands with no crossings lying to the right of the braid, viewing the braid arranged from top to bottom. 

Let $r_1, \dots, r_k$ and $s_i, \dots, s_k$ be integers such that
$2\leq r_1< \dots < r_k$, and $s_i>0$ for all $i$. The T-link $T((r_1,s_1), \dots, (r_k,s_k))$ is defined to be the closure of the braid
\[ (\sigma_1\sigma_2\dots\sigma_{r_1-1})^{s_1}(\sigma_1\sigma_2\dots\sigma_{r_2-1})^{s_2}\dots(\sigma_1\sigma_2\dots\sigma_{r_k-1})^{s_k}.\]
Thus $T((r_1,s_1),\dots,(r_k,s_k))$ is obtained by concatenating the braids $(r_i,s_i)$ within the braid group on $r_k$ strands, and then taking the closure.

Finally, a \emph{full twist} on $p$ strands is given by
  \[ \Delta^2 = (\sigma_1\dots \sigma_{p-1})^p. \]

\begin{proposition}\label{Prop:BraidMove}
Let $p$, $q$, and $r$ be positive integers with $0<q\leq r<p$. 
There is an ambient isotopy of $S^3$ taking the $(p,q)$ torus link to the closure of the braid on $r$ strands given by
\[ (\sigma_{r-1}\dots \sigma_{r-q+1})^{p-r}(\sigma_{1}\dots \sigma_{r-1})^{q}. \]

Moreover, an ambient isotopy realising the equivalence fixes the portion of the braid $(\sigma_{1}\dots \sigma_{p-1})^{q}$ corresponding to the $r$ left-most strands at the top the braid. Thus, we may replace a neighbourhood of these strands above the braid $(\sigma_{1}\dots \sigma_{p-1})^{q}$ with any tangle $\tau$ on $r$ strands, and we find that the resulting link is ambient isotopic to the closure of the link obtained by concatenating the braid on $r$ strands $(\sigma_{r-1}\dots \sigma_{r-q+1})^{p-r}$, with $\tau$, and then with $(\sigma_{1}\dots \sigma_{r-1})^{q}$. See \reffig{TorusBraidMove}.
\end{proposition}

\begin{figure}
  \includegraphics{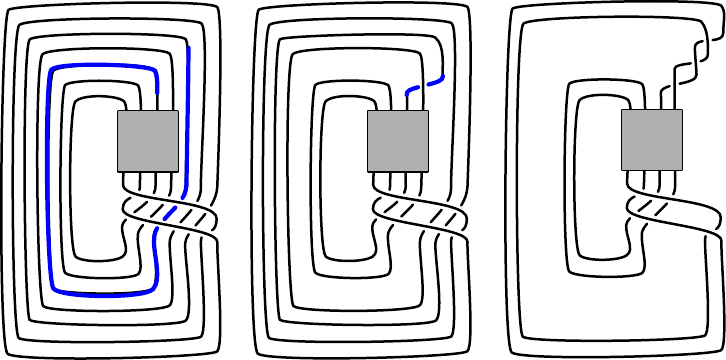}
  \caption{Illustration of \refprop{BraidMove} in the case that $q=2$, $r=4$, $p=7$, for an arbitrary tangle shown as a gray box. The left-most picture shows the original link. The $(r+1)$-st strand, shown in blue, can be pulled tight beneath the diagram, resulting in the middle picture. The right-most picture shows the result after isotoping strands $(r+1)$ to $p$.}
  \label{Fig:TorusBraidMove}
\end{figure}

\begin{proof}
Because $r\geq q$, the $(r+1)$-st strand at the top of the braid only runs under the $q$ overcrossing strands in the braid corresponding to the $(p,q)$ torus link.  It then runs around the braid closure back to the top, returning to the $r-q+1$ position. Together with a horizontal line from the $r-q+1$ position to the $r+1$ position, this strand bounds a disc in $S^3$, lying under the plane of projection. Use this disc to push the strand in $S^3$ to become a horizontal strand lying below the plane of projection, running from the $r+1$ position, then behind $q$ strands, to the $r+1-q$ position. Adjust slightly, pulling the right side up, so that the result is a closed braid; see \reffig{TorusBraidMove}, middle. Note that the resulting braid consists of only $p-1$ strands. This isotopy generalises the isotopy given by Lee in~\cite[Figure~6]{Lee:Positively}, and by de~Paiva in~\cite[Figure~1]{de2022torus}. 

This move can be repeated for all the $p-r$ strands to the right of the $(r+1)$-st strand. When finished, we obtain a link on $r$ strands as claimed. 
\end{proof}

\subsection{Braid index}

Recall that the \emph{braid index} of a knot $K$, which we will denote $\beta(K)$, is the minimal number of strands required to form a braid with closure isotopic to $K$. We will repeatedly use the following result of Franks and Williams~\cite{FranksWilliams} on braid index of the closure of a positive braid. 

\begin{theorem}[Corollary~2.4 of \cite{FranksWilliams}]\label{Thm:FranksWilliams}
  Let $B$ be a positive braid on $p$ strands that contains a full twist $\Delta^2$. 
  Then $B$ has braid index $p$. \qed
\end{theorem}

\begin{lemma}\label{Lem:pqdrBraidIndex}
Let $p$, $q$, $d$ and $r$ be positive integers such that $q\leq r <p$ and $d+q\geq r$. Let $B_r$ be a positive braid on $r$ strands, and let $B_p$ denote the braid on $p$ strands obtained by adding $p-r$ trivial strands to the right of the braid $B_r$. 
Then the closure of the braid on $p$ strands
\[ B_p (\sigma_1\dots \sigma_{r-1})^d(\sigma_1\dots \sigma_{p-1})^q \]
has braid index equal to $r$. 
\end{lemma}

\begin{proof}
By \refprop{BraidMove}, the closure of the given braid on $p$ strands is equivalent to the closure of the braid on $r$ strands
\[ B'=(\sigma_{r-1}\dots \sigma_{r-q+1})^{p-r} B_r (\sigma_1\dots \sigma_{r-1})^d(\sigma_{1}\dots \sigma_{r-1})^q.
\]
Because this is a positive braid, and because $d+q\geq r$, the braid $B'$ has at least one positive full twist on $r$ strands. Thus \refthm{FranksWilliams} implies that the closure of $B$ (and $B'$) has braid index equal to $r$.
\end{proof}

\begin{corollary}\label{Cor:TLinkBraidIndex}
Suppose $0<r_1<\dots < r_n < p$ are integers, $s_1, \dots, s_n$ and $q$ are positive integers, and suppose $q\leq r_n \leq s_n+q$. Then the T-link
\[  K = T((r_1, s_1), \dots, (r_n, s_n), (p, q)) \]
has braid index equal to $r_n$.
\end{corollary}

\begin{proof}
Let $B_{r_n}$ be the braid on $r_n$ strands obtained as the concatenation of torus braids $(r_1,s_1) \dots(r_{n-1},s_{n-1})$, where we view each $(r_i,s_i)$ as a braid on $r_n$ strands by adding $r_n-r_i$ trivial strands to the right of the braid $(r_i,s_i)=(\sigma_1\dots\sigma_{r_i-1})^{s_i}$. Then the given T-link is the closure of the braid
\[ B_{r_n} (\sigma_1\dots\sigma_{r_n-1})^{s_n}(\sigma_1\dots\sigma_{p-1})^q. \]
Since $q\leq r_n \leq s_n+q$, the result follows from \reflem{pqdrBraidIndex}. 
\end{proof} 

The next definition is from Williams~\cite{Williams:GenCables}. 

\begin{definition}
A \emph{generalized $q$-cabling} of a link $L$ is a link $L'$ contained in the interior of a tubular neighbourhood $L\times D^2$ of $L$ such that
\begin{enumerate}
\item each fiber $D^2$ intersects $L'$ transversely in $q$ points;  and
\item all strands of $L'$ are oriented in the same direction as $L$ itself.
\end{enumerate} 
\end{definition}

Williams showed the following result on generalised $q$-cablings for knotted $L$ in~\cite{Williams:GenCables}. 

\begin{theorem}[Theorem~1 of Williams \cite{Williams:GenCables}]\label{Thm:Williams}
The braid index is multiplicative under generalized cabling. That is, if $L$ is a link with each component a non-trivial knot and $L'$ is a generalized $q$-cabling of $L$ then $\beta(L') = q\beta(L),$ where $\beta(*)$ is the braid index of $*$. \qed
\end{theorem}

This result was extended to unknotted $L$ in the case of positive braids by de~Paiva in~\cite{depaiva2021hyperbolic}. The following result is from that paper. 

\begin{lemma}[Lemma~2.3 of \cite{depaiva2021hyperbolic}]\label{Lem:WilliamsTrivialCase}
Let $L'$ be a generalized $q$-cabling of the unknot $L$, with $L$ given by a positive braid on $n$ strands, where $n>1$. Also, assume the knot inside $L$ is given by a positive braid. Then $L'$ has braid index equal to $q$. \qed
\end{lemma}

\section{Parents of T-links}

In this section, we build the ``parent links'' mentioned in the introduction. Dehn filling on such links produces T-links with full twists. By classifying when such links are hyperbolic, and applying Thurston's hyperbolic Dehn filling theorem, we show that, in an appropriate sense, most T-links with only full twists are hyperbolic. This is an extension of work by de~Paiva and Purcell~\cite{depaiva-purcell2021satellites}. There, the same links were constructed, and some conditions were given to guarantee hyperbolicity. Here, we strengthen the result by completely characterising when such links are hyperbolic.

\begin{definition}[The unknots $J_a$]\label{Def:Unknots}
Let $p, q$ be relatively prime integers such that $1<q<p$. Consider the $(p,q)$-torus braid on $p$ strands, and its closure, the torus link $T(p,q)$. Let $F$ denote the Heegaard torus on which $T(p, q)$ lies.
Let $a$ be an integer with $0< a < p$. Define the \emph{unknot $J_{a}$} to be an unknot lying horizontally with respect to the $(p,q)$-torus braid, positioned just above the crossings of the braid, bounding a disc such that the interior of that disc meets $F$ transversely in a single arc intersecting the $a$ leftmost strands of the braid. 

More generally, given $a_1, \dots, a_n$ satisfying $1< a_1 < \dots < a_n <p$, define \emph{unknots $J_{a_{1}}, \dots, J_{a_n}$} to be disjoint unknots as above, positioned so that the $i$-th is pushed vertically above the $(i+1)$-th with respect to the braid, so that all are disjoint. Figure~\ref{Fig:AugmentedLink} shows an example.
\end{definition}

These links $T(p,a)\cup J_{a_1}\cup \dots \cup J_{a_n}$ can be Dehn filled to obtain T-links. They are also studied in \cite{depaiva-purcell2021satellites}, where it is shown that $K$ is hyperbolic whenever all the $a_i>q$ are not multiples of $q$. Below, we improve that result to show that only one of the $a_i$ need not be a multiple of $q$ to guarantee hyperbolicity. 

\begin{figure}
  \includegraphics{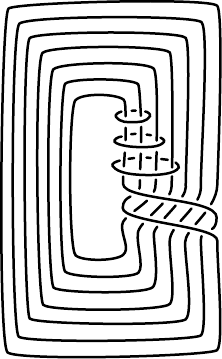}
  \caption{Shows $T(7,2)$ augmented at the top right by $J_2$, $J_3$, and $J_4$.}
  \label{Fig:AugmentedLink}
\end{figure}

\begin{proposition}\label{Prop:Atoroidal}
Let $p, q$ be relatively prime integers with $1 < q < p$. Let $a_1, \dots, a_n$ be integers such that $1< a_1 < \dots < a_n <p$, with $n>1$.
Finally, assume that either there is $a_i>q$ that is not a multiple of $q$, or that $a_n=q$. Then the link 
$K = T(p, q)\cup J_{a_n}\cup \dots\cup J_{a_1}$ is atoroidal.
\end{proposition}

\begin{proof}
Suppose $S^3-N(K)$ admits an essential torus $T$. Then $T$ bounds a solid torus $V \subset S^3$ that must contain at least one component of $K$.

First we show that we may choose $V$ to contain $T(p,q)$.
For suppose $V$ is disjoint from $T(p,q)$. Then it must contain at least one $J_{a_j}$. The component $J_{a_j}$ must have positive wrapping number in $V$, for otherwise $T(p,q)$ and $J_{a_j}$ would have zero linking number, which is a contradiction. 
Because there is no essential torus in the exterior of the unknot in $S^3$, it follows in this case that $T$ is unknotted in $S^3$. Therefore, $T$ bounds a second solid torus $V'$ containing $T(p,q)$. Thus in all cases we may assume $T$ bounds a solid torus containing $T(p,q)$.

As $a_n \geq q$, by \refprop{BraidMove}, the torus knot $T(p, q)$ is isotopic to a closed braid with $a_n$ strands so that under the isotopy, the largest unknot $J_{a_n}$ becomes the braid axis. Because the isotopy moves only the right-most $p-a_n$ strands, all unknots $J_{a_1}, \dots, J_{a_n}$ are untouched by the isotopy.

The torus $T$ is then contained in the solid torus $S^3-N(J_{a_n})$, and bounds a solid torus $V$ containing $T(p,q)$. It follows that $J_{a_n}$ is disjoint from $V$. 

The torus $T$ must intersect the disc $D_{a_n}$ bounded by $J_{a_n}$ in a series of circles, with each circle bounding a meridian of $V$. Each meridian of $V$ can be isotoped to meet the same number of strands of $T(p,q)$, as follows. The boundary of a meridian of $V$ defines an unknot in $S^3$, and all such unknots are isotopic in $S^3-N(K)$, where the isotopy is obtained by pushing the boundary of the meridian of $V$ along the torus $T=\bdy V$. Because $T(p,q)$ forms a braid in the complement of $J_{a_n}$, it meets these discs monotonically.

Let $b$ denote the number of times that a meridian of $V$ intersects the strands of $T(p,q)$ on the disc $D_{a_n}$. Note $b>1$, or else $T$ would be boundary parallel to $T(p, q)$.
Note also that $b<a_n$, otherwise $T$ would be boundary parallel to $J_{a_n}$.

Observe that $V$ winds some number of times around the solid torus $S^3-N(J_{a_n})$, and note that each meridian of the latter solid torus meets exactly $a_n$ strands of $T(p,q)$, since this is the number of strands in the closed braid isotopic to $T(p,q)$ obtained from \refprop{BraidMove}.
It follows that $b$ must divide $a_n$.
Then $T(p,q)$ is a generalised $b$-cabling of $L$, where $L$ is the core of the solid torus $V$.

Observe that $T$ embeds in the exterior of the torus knot $S^3-N(T(p,q))$. By work of Tsau~\cite{Tsau},
there are no essential tori in a torus knot exterior. Because $b>1$, it follows that $T$ must be compressible to its outside. That is, $V$ is unknotted in $S^3$. Thus, \reflem{WilliamsTrivialCase} implies that $T(p,q)$ has braid index equal to $b$.

On the other hand, the torus knot $T(p,q)$ with $1<q<p$ has braid index equal to $q$; for example this follows from Franks and Williams' \refthm{FranksWilliams}. Then $b=q$. Since $b$ divides $a_n$, it follows that $q$ divides $a_n$. Recall also that we must have a strict inequality $b< a_n$. This rules out essential tori in the case $a_n=q$. Thus we many now assume that $a_n>q$.

By hypothesis, there is $a_i \in \lbrace a_1, \dots, a_n \rbrace$ which is greater than $q$ and not a multiple of $q$.
Since $a_i>q$, it must be the case that $J_{a_i}$ is disjoint from the solid torus $V$. Since $T(p,q)$ intersects the disc $D_{a_i}$ bounded by $J_{a_i}$ a total of $a_i$ times, and $T(p,q)$ is a generalised $q$-cabling of $L$, it must be the case that $L$ intersects the disc $a_i/q$ times. However, $q$ does not divide $a_i$. This is a contradiction. 
\end{proof}

\begin{lemma}\label{Lem:TwoComponents}
Let $p, q$ be relatively prime integers with $1 < q < p$. Let $a_n, \dots, a_1$ be integers such that $1< a_1 < \dots < a_n <p$  with $n>1$.
Then the link
$K = T(p, q)\cup J_{a_n}\cup \dots\cup J_{a_1}$
has no annuli with boundaries in two different components.
\end{lemma}

\begin{proof}
Suppose that $S^3 - N(K)$ has an annulus $A$ with boundaries $\partial_1 A$ and $\partial_2 A$ that lie in two different components, $C_1$ and $C_2$, respectively, of $\partial (S^3 - N(K))$.

\vspace{1mm}

\textbf{Case 1:} 
Consider first that $C_1$ and $C_2$ are $J_{a_j}$ and $J_{a_k}$, respectively, for some $j\neq k \in \{1, \dots, n\}$. 

Note $\bdy_1 A$ and $\bdy_2 A$ are isotopic in $S^3-N(K)$.
The linking number between $C_j$ and $\bdy_j A$ is zero if and only if $\bdy_j A$ is the longitude of $C_j$, in which case $C_j$ and $\bdy_j A$ are isotopic, for $j=1,2$.

Suppose $\bdy_1 A$ is the longitude of $C_1$, but $\bdy_2 A$ is not the longitude of $C_2$. Since $\bdy_1 A$ and $\bdy_2 A$ are isotopic, $C_1$ and $C_2$ would have nonzero linking number in this case, but this is not possible. Similarly $\bdy_2 A$ cannot be the longitude of $C_2$ if $\bdy_1 A$ is not the longitude of $C_1$.

Thus either $\bdy_1 A$ is the longitude of $C_1$ and $\bdy_2 A$ is the longitude of $C_2$, or neither is a longitude. If both are longitudes, then $C_1$ and $C_2$ are isotopic, which is not possible. Thus neither are longitudes. 

Then the linking number between $C_2$ and $\bdy_2 A$ is positive. However, $C_1$ and $C_2$ have zero linking number, so $\bdy_1 A$ and $C_2$ must have zero linking number. But $\bdy_2 A$ is isotopic to $\bdy_1 A$, and so $\bdy_1 A$ and $C_2$ have nonzero linking number equal to the linking number of $C_2$ and $\bdy_2 A$. This is a contradiction.

\vspace{1mm}

\textbf{Case 2:}
Now suppose that $C_1$ and $C_2$ are $J_{a_j}$ and $T(p, q)$, respectively, for some $j\in \{1, \dots, n\}$. Again $\bdy_1 A$ and $\bdy_2 A$ are isotopic. 

Suppose first that $\bdy_2 A$ wraps at least one time along the longitude of $C_2=T(p,q)$. Then $\bdy_2 A$ has positive linking number with each of the components $J_{a_k}$, because $T(p,q)$ has positive linking number with each. But the linking number between $\bdy_2 A$ and $J_{a_k}$ for $J_{a_k} \neq C_1$ is zero, because $C_1$ has linking number zero with each such component, and $\bdy_2 A$ has the same linking number with $C_1$ as $\bdy_1 A$. This is a contradiction.

Thus $\bdy_2 A$ is a meridian of $C_2=T(p,q)$. So $\bdy_2 A$ and $T(p,q)$ have linking number equal to one. The curve $\bdy_1 A$ is some torus knot $T(a, b)$ on $\bdy N(C_1)$. If $a$ is equal to  zero, then $\bdy_1 A$ is a meridian of $C_1$. Because a meridian of $C_1$ has linking number zero with $C_2=T(p,q)$, it follows that $\bdy_1 A$ and $T(p,q)$ have linking number equal to zero. However, this is not possible as $\bdy_1 A$ and $\bdy_2 A$ are isotopic. So, $a\neq 0$. The linking number between $\bdy_1 A$ and $C_2=T(p,q)$ is equal to $a\cdot a_j$, where $C_1=J_{a_j}$. Because $\bdy_2 A$ and $T(p,q)$ have linking number $1$, and $\bdy_1 A$ and $T(p,q)$ have linking number identical to $\bdy_2 A$ and $T(p,q)$, it follows that $a\cdot a_j = 1$. This is impossible since $a_j>1$. Therefore, no such annulus exists. 
\end{proof}

\begin{lemma}\label{Lem:AnnulusT(p,q)}
Let $K$ be as in \refprop{Atoroidal}. Then $S^3-N(K)$ has no essential annuli with both boundary components in $\partial N(T(p,q))$.
\end{lemma}

\begin{proof}
Suppose that $S^3 - N(K)$ has an essential annulus $A$ with both boundary components in $\partial N(T(p,q))$. Either $A$ is essential in $S^3-N(T(p,q))$, or inessential.
The exterior of a torus knot has just one essential annulus by work of Tsau~\cite{Tsau}. 

Suppose first that $A$ is the essential annulus in $S^3-N(T(p,q))$. If there is $a_i>q$ that is not a multiple of $q$, then
by work of Lee, \cite[Lemma~5.1]{Lee:Unknotted} that essential annulus would be punctured by $J_{a_i}$, where $a_i>q$ is not a multiple of $q$. This would be a contradiction. 

So suppose that $a_n=q$ and $n>1$.
The unique essential annulus in $S^3-T(p,q)$ lies on the trivial torus $T\subset S^3$ onto which $T(p,q)$ is projected, and $A=T-T(p,q)$.
Each component $J_{a_i}$ is disjoint from $A$ and from $T(p,q)$, so it must lie in one of the solid tori on either side of $T$. Denote by $\omega_i$ the winding number of $J_{a_i}$ in the solid torus of $\mathbb{S}^3-T$ in which it lies. Then the linking number of $J_{a_i}$ and $T(p, q)$ is equal to $p\omega_i$ or $q\omega_i$. But that linking number is known to be $a_i$, $i=1, \dots, n$, and $a_n=q$. Since there exists $a_i<a_n=q<p$, this implies the linking number of $J_{a_i}$ and $T(p,q)$ is zero, which is a contradiction.

Thus $A$ is not essential in $S^3-N(T(p,q))$, so it must be compressible, boundary compressible, or boundary parallel in $S^3-N(T(p,q))$. Observe that a boundary compressible annulus is in fact boundary parallel, using the fact that $S^3-N(T(p,q))$ is irreducible and boundary irreducible.
 
Consider first that $A$ is boundary parallel to an annulus $B$ in $\partial N(T(p,q))$. Then $A\cup B$ bounds a solid torus $V$ in $S^3-N(T(p,q))$. Since $A$ is not boundary parallel in $S^3-N(K)$, at least one $J_{a_j}$ must be inside $V$. In addition, $J_{a_j}$ has wrapping number greater than zero in $V$, or else $T(p, q)$ and $J_{a_j}$ would have linking number equal to zero, which is a contradiction. But $J_{a_j}$ is an unknot, whose complement admits no essential tori (e.g.\ \cite[page~15]{hatcher2007notes}). Thus $V$ is also unknotted in $S^3$. This implies that $B$ is a meridional annulus of $\partial N(T(p,q))$. If $\partial V$ is boundary parallel to $J_{a_j}$, then $J_{a_j}$ is the core of $\partial V$. Hence, the linking number between $T(p,q)$ and $J_{a_j}$ would be one, which is not possible. Thus, as $\partial V$ is not boundary parallel to $J_{a_j}$, $\partial V$ is an essential torus for $S^3 - N(K)$. This contradicts \refprop{Atoroidal}.

Assume now that $A$ is compressible in $S^3-N(T(p,q))$. Then there is a compression disk $D$ for $A$ in $S^3-N(T(p,q))$. Surgering $A$ along $D$ yields two discs, $D_1$ and $D_2$, such that $\partial A = \partial D_1 \cup \partial D_2$.
Since $S^3-N(T(p,q))$ is boundary irreducible, $\partial D_i$ bounds a disk $E_i$ on $\partial N(T(p,q))$. Thus, by pushing $E_i$ slightly off of $\bdy N(T(p, q))$ in $S^3-N(K)$, we obtain a compressing disc for $A$ in $S^3-N(K)$, which contradicts our assumption that $A$ is essential. Therefore, $A$ is not compressible. 

Thus $A$ cannot have both boundary components on $\bdy N(T(p,q))$.
\end{proof}

\begin{lemma}\label{Lem:AnnulusJ_{a_i}}
  Let $K$ be as in \refprop{Atoroidal}. Then $K$
has no essential annulus with both boundary components on one $\bdy N(J_{a_j})$.
\end{lemma}

\begin{proof}
Suppose that $S^3-N(K)$ has an essential annulus $A$ with both boundary components on $\partial N(J_{a_j})$. Since $S^3-N(J_{a_j})$ is a solid torus, and the solid torus admits no essential annuli, $A$ is not essential in $S^3-N(J_{a_j})$. Thus $A$ is either compressible or boundary parallel in $S^3-N(J_{a_j})$.

\vspace{2mm}

\textbf{Case A: }
Suppose $A$ is boundary parallel, parallel to an annulus $B$ in $\bdy N(J_{a_j})$. Then $A\cup B$ bounds a solid torus $V$ in $S^3-N(J_{a_j})$. Since $A$ is not boundary parallel in $S^3-N(K)$, at least one component $C$ of $K$ must be inside $V$. 

\textbf{Case A1: }
Consider first that $C = T(p, q)$. Then $T(p, q)$ has wrapping number greater than zero in $V$, for otherwise $J_{a_j}$ and $T(p, q)$ would have zero linking number, a contradiction. Note this implies that $\bdy V$ is incompressible to its inside. 

Suppose that some circle $J_{a_k}$ with $j\neq k$ lies in $S^3-V$. Then we may isotope $J_{a_k}$ to lie outside of $W = N(J_{a_j})\cup V$, which is a regular solid torus neighbourhood of the unknot $J_{a_j}$.
Denote by $\omega$ the winding number of $J_{a_k}$ in $S^3-W$. If $\omega = 0$, then the linking number between $J_{a_k}$ and $T(p, q)$ is zero. Thus, $\omega \neq 0$.
But then this implies that the linking number between $J_{a_j}$ and $J_{a_k}$ is nonzero, a contradiction. 
Thus all circles $J_{a_{1}}, \dots, J_{a_{i-1}}, J_{a_{i+1}}, \dots, J_{a_n}$ are inside $V$ in this case.
Because at least two components of $K$ lie inside $V$, $\bdy V$ is not boundary parallel to the inside. 

The core of $V$ forms a torus knot $T(a,b)$ on $N(J_{a_j})$. Note $b>0$ or else $T(p,q)$ runs around a longitude of $N(J_{a_j})$ and hence has linking number zero with $J_{a_j}$, a contradiction.

Suppose $b=\pm 1$, so the core of $V$ has the form of the trivial knot $T(a, \pm 1)$. Then there exists a disc $D$ in $S^3-N(K)$ that is a longitude for $\bdy N(J_{a_j})$ whose boundary $\bdy D$ can be divided into two arcs, one of which meets $A\subset \bdy V$ in a nontrivial arc, and the other meets $\bdy N(J_{a_j})$. See \reffig{BoundaryCompression}.
This is an essential boundary compression disc for $A$, contradicting the fact that $A$ is essential.

\begin{figure}
\begingroup%
  \makeatletter%
  \providecommand\color[2][]{%
    \errmessage{(Inkscape) Color is used for the text in Inkscape, but the package 'color.sty' is not loaded}%
    \renewcommand\color[2][]{}%
  }%
  \providecommand\transparent[1]{%
    \errmessage{(Inkscape) Transparency is used (non-zero) for the text in Inkscape, but the package 'transparent.sty' is not loaded}%
    \renewcommand\transparent[1]{}%
  }%
  \providecommand\rotatebox[2]{#2}%
  \newcommand*\fsize{\dimexpr\f@size pt\relax}%
  \newcommand*\lineheight[1]{\fontsize{\fsize}{#1\fsize}\selectfont}%
  \ifx\svgwidth\undefined%
    \setlength{\unitlength}{129.66864395bp}%
    \ifx\svgscale\undefined%
      \relax%
    \else%
      \setlength{\unitlength}{\unitlength * \real{\svgscale}}%
    \fi%
  \else%
    \setlength{\unitlength}{\svgwidth}%
  \fi%
  \global\let\svgwidth\undefined%
  \global\let\svgscale\undefined%
  \makeatother%
  \begin{picture}(1,0.58589374)%
    \lineheight{1}%
    \setlength\tabcolsep{0pt}%
    \put(0,0){\includegraphics[width=\unitlength,page=1]{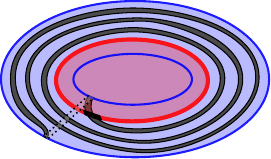}}%
    \put(0.36860258,0.22943075){\color[rgb]{0,0,0}\makebox(0,0)[lt]{\lineheight{1.25}\smash{\begin{tabular}[t]{l}$A\cap \bdy D$\end{tabular}}}}%
    \put(0,0){\includegraphics[width=\unitlength,page=2]{LongitudinalDiscUpdate.pdf}}%
  \end{picture}%
\endgroup%

  \caption{The blue torus is the trivial torus $\bdy N(J_{a_j})$. The solid torus $V$ is shown in dark gray, embedded here as a $(1,3)$-torus knot, which is an unknot. The disc $D$ is shown in red. Its boundary is a longitude, with the solid red arc on $\bdy D\cap \bdy N(J_{a_j})$ and the short black arc on $(A\subset \bdy V) \cap \bdy D$. }
  \label{Fig:BoundaryCompression}
\end{figure}

Since $|b|>1$, $\bdy V= \bdy N(T(a,b))$ is incompressible and not boundary parallel to the outside, i.e.\ in the solid torus $S^3-J_{a_j}$. 

This implies that in all cases $\bdy V$ is essential in $S^3-N(K)$ contradicting \refprop{Atoroidal}.

\textbf{Case A2: }
The torus knot $T(p, q)$ cannot lie inside $V$ by the previous case. So some $C=J_{a_k}$ with $j\neq k$ lies inside $V$. The wrapping number of $J_{a_k}$ inside $V$ must be different from zero as $J_{a_k}$ and $T(p, q)$ have positive linking number. Since $J_{a_k}$ and $J_{a_j}$ have zero linking number, $V$ must be a longitude of $\bdy N(J_{a_j})$. If $J_{a_k}$ is the core of $V$, then $J_{a_j}$ and $J_{a_k}$ are isotopic in $S^3-N(T(p,q))$, a contradiction. So $J_{a_k}$ is not the core of $V$. But then $\bdy V$ is incompressible and not boundary parallel to the inside in $S^3-K$, and incompressible and not boundary parallel to the outside in $S^3-K$, contradicting \refprop{Atoroidal}. 

\vspace{1mm}

\textbf{Case B: }
Suppose $A$ is compressible in $S^3-N(J_{a_j})$. Then there is a compression disk $D$ for $A$ in $S^3-N(J_{a_j})$. Surgering $A$ along $D$ yields two discs, $D_1$ and $D_2$, such that $\partial A = \partial D_1 \cup \partial D_2$.
If one of $\bdy D_1$ or $\bdy D_2$ bounds a disk $E$ on $\bdy N(J_{a_j})$, then by considering a disc with boundary in $A$ close to $E$, we see that $A$ is also compressible in $S^3-N(K)$, a contradiction. So suppose that neither $\bdy D_1$ nor $\bdy D_2$ bounds a disk on $\bdy N(J_{a_j})$. Then $D_1$ and $D_2$ are discs in the solid torus $S^3-N(J_{a_j})$ with nontrivial boundary on $\bdy N(J_{a_j})$ and hence both are meridians of $S^3-N(J_{a_j})$, i.e.\ with $\bdy D_1$ and $\bdy D_2$ forming longitudes of $\bdy N(J_{a_j})$.
Undoing the surgery along $D$, it follows that $A$ is boundary parallel in $S^3-N(J_{a_j})$. Thus we have a contradiction to Case A. 

Therefore, $S^3-N(K)$ has no essential annulus with both boundary components in one $\partial N(J_{a_j})$.
\end{proof}

\begin{proposition}\label{Prop:NoEssentialAnnuli}
  The link $K$ as in \refprop{Atoroidal}
has no essential annuli.
\end{proposition}

\begin{proof}
By \reflem{TwoComponents}, any essential annulus has both boundary components on the same component of $K$. By \reflem{AnnulusT(p,q)}, the two boundary components cannot lie on $\bdy N(T(p,q))$. By \reflem{AnnulusJ_{a_i}} the two boundary components cannot lie on one of the $\bdy N(J_{a_j})$. Thus no such annulus exists. 
\end{proof}

\begin{theorem}\label{Thm:HyperbolicAugLink}
Let $p, q$ be relatively prime integers with $1 < q < p$. Let $a_n, \dots, a_1$ be integers such that $1< a_1 < \dots < a_n <p$ with $n>1$.
Finally, assume that either there is $a_i>q$ which is not a multiple of $q$, or that $a_n=q$. 
Then the link $K= T(p, q)\cup J_{a_n}\cup \dots\cup J_{a_1}$ is hyperbolic.
\end{theorem} 

\begin{proof}
By de~Paiva and Purcell \cite[Lemma~5.1]{depaiva-purcell2021satellites}, the link exterior is irreducible and boundary irreducible. By \refprop{Atoroidal}, it is atoroidal. By \refprop{NoEssentialAnnuli}, it is anannular. Therefore it is hyperbolic by Thurston's hyperbolization theorem for Haken manifolds~\cite{thurston:bulletin}. 
\end{proof}

Combining \refthm{HyperbolicAugLink} and de~Paiva and Purcell \cite[Theorem~5.6]{depaiva-purcell2021satellites}, we completely classify the geometric types of the links $T(p,q)\cup J_{a_1}\cup \dots J_{a_n}$.

\begin{theorem}\label{Thm:AugHyperbolic}
Fix relatively prime integers $1<q<p$, and integers $1<a_1< \dots< a_n<p$. Then the link
$K = T(p,q)\cup J_{a_1}\cup \dots J_{a_n}$
is hyperbolic if and only if one of the following holds:
\begin{itemize}
\item either all $a_i<q$, or
\item $a_n=q$ and $n>1$, or
\item there is $a_i>q$ which is not a multiple of $q$.
\end{itemize}
\end{theorem}

\begin{proof}
When $n=1$, the link $K=T(p,q)\cup J_{a_1}$ is the Dehn-filling parent of a twisted torus knot; this has been treated by Lee~\cite{Lee:Cable, Lee:Unknotted}. If $n=1$ and $a_1 = q$, then~\cite[Theorem~1]{Lee:Cable} implies that infinitely many Dehn surgeries along $J_{a_1}$ yield non-hyperbolic knots. Therefore, Thurston's hyperbolic Dehn filling theorem~\cite{thurston:notes} implies $K$ is not hyperbolic. In fact, the proof of~\cite[Theorem~1]{Lee:Cable} implies $K$ is annular. If $n=1$ and $a_1$ is a multiple of $q$, then $K$ is a satellite link, by \cite[Proposition~3.2]{depaiva-purcell2021satellites}.
If $n=1$ and $a_1$ is not a multiple of $q$, then $K$ is hyperbolic by~\cite[Proposition~5.7]{Lee:Unknotted}. This proves the result in the case $n=1$. 

In the case $n>1$, if there is $a_i>q$ that is not a multiple of $q$, or if $a_n=q$, then $K$ is hyperbolic by \refthm{HyperbolicAugLink}.
If $n>1$ and all $a_i$ are less than $q$, then no $a_i$ is a multiple of $q$, and $K$ is hyperbolic by \cite[Theorem~5.6]{depaiva-purcell2021satellites}.
Finally, for the converse, if $n>1$, and there is some $a_i>q$, and all $a_i>q$ are multiples of $q$, then $K$ is satellite by \cite[Theorem~5.6]{depaiva-purcell2021satellites}.
\end{proof}

\begin{corollary}\label{Cor:FullTwistConj}
Let $p,q$ be relatively prime integers with $1<q<p$, and let $a_1, \dots, a_n$ and $s_1, \dots, s_n$ be integers such that $1<a_1<\dots<a_n<p$ and $s_i>0$ for all $i$. 
Then there exists $B\gg 0$ such that if each $s_i>B$, the T-link
\[ T((a_1, a_1s_1), \dots, (a_n,a_ns_n), (p,q)) \]
is hyperbolic if and only if one of the following holds:
\begin{itemize}
\item all $a_i< q$, or
\item $a_n=q$ and $n>1$, or
\item there is $a_i> q$ which is not a multiple of $q$.
\end{itemize}
\end{corollary}

\begin{proof}
By \refthm{AugHyperbolic}, the link
$K=T(p,q)\cup J_{a_1} \cup \dots \cup J_{a_n}$
is hyperbolic if and only if the $a_i$ satisfy the hypotheses of the corollary. Obtain the given T-link by Dehn filling the link components $J_{a_1}, \dots, J_{a_n}$ along slopes $1/s_1, \dots, 1/s_n$, respectively. Suppose first that the link $K$ is hyperbolic. Then Thurston's hyperbolic Dehn filling theorem~\cite{thurston:notes} implies that if the $s_i$ are sufficiently large, then Dehn filling yields a hyperbolic T-link. Now suppose $K$ is not hyperbolic. If $K$ is a satellite link, then Dehn filling $K$ yields a satellite T-link, by de~Paiva and Purcell~\cite[Theorem~5.6]{depaiva-purcell2021satellites}. Otherwise $K$ is a torus link, $n=1$ and $a_1=q$, and Dehn filling yields a torus link by Lee~\cite{Lee:Cable}. 
\end{proof}

Note that \refthm{MainGeneric} in the introduction follows immediately from \refcor{FullTwistConj}. 

\section{Hyperbolicity with effective full twist bounds}\label{Sec:SomeParameters}

While \refcor{FullTwistConj} is quite broad, unfortunately the constant $B$ in that theorem is not explicit, and so it may be difficult to apply in practice. In this section we find explicit parameters which produce hyperbolic T-knots obtained by full twists.

\begin{proposition}\label{Prop:2Twists_an_p_RelPrime}
Let $a_1, \dots, a_n$, $s_1, \dots, s_n$, and $p, q, k$ be integers satisfying the following hypotheses:
\begin{itemize}
\item $p$ and $q$ are relatively prime,
\item $1<a_1<\dots < a_n$, and $0 < q < a_n < p$, 
\item each $s_i>0$, and $s_n\geq 2$,
\item $p$ and $a_n$ are relatively prime,
\item $k\geq 2$. 
\end{itemize}
Then the T-knot
$K=T((a_1, a_1s_1), \dots, (a_n, a_ns_n), (p, q + kp))$
is atoroidal. 
\end{proposition}

\begin{proof}
Suppose that the exterior of $K$ in $S^3$ admits an essential torus $T$. By work of Ito~\cite[Theorem~1.2(3)]{Ito}, because $K$ is the closure of a braid with at least two positive full twists on $p$ strands, the torus $T$ does not intersect the braid axis $C$. Moreover, the knot inside $T$ is given by a braid. Thus there exists some integer $d>1$ such that $K$ is a generalized $d$-cabling of a knot $L$, where $L$ is the core of the solid torus bounded by $T$. As a consequence, $d$ must divide $p$. 

After $(-1/k)$-Dehn surgery along the braid axis $C$, the knot $K$ becomes the T-knot 
\[ K' = T((a_1, a_1s_1), \dots, (a_n, a_ns_n), (p, q)) \]
and the torus $T$ becomes a new torus $T'$. This will bound a solid torus $V'$ in $S^3$, with core $L'$. Because $q < a_n < a_ns_n+q$, the knot $K'$ has braid index equal to $a_n$ by \refcor{TLinkBraidIndex}.

If $L'$ is trivial, then $a_n$ is equal to $d$ by \reflem{WilliamsTrivialCase}. However, this is not possible since $\gcd(p, a_n) = 1$.

So $L'$ is knotted. Then by \refthm{Williams}, $a_n$ is equal to $d\beta(L')$, where $\beta(L')$ is the braid index of $L'$. But then $d$ divides $p$ and $d$ divides $a_n$, again contradicting $\gcd(p,a_n)=1$. 

Therefore, the exterior of $K$ admits no essential torus.
\end{proof}

We will combine the previous result with the following from~\cite{de2022torus}, which gives information on torus knots.

\begin{theorem}[Theorem~1.2 of de~Paiva~\cite{de2022torus}] \label{Thm:NonTorusKnots}
Let $p,q, a_1, \dots, a_n, s_1, \dots, s_n$ be positive integers such that $1<q<p$ and $1<a_1<\dots<a_n<p$ with $a_i\neq q$.
If $\gcd(p,q)=1$ and in addition one of the following hold:
\begin{itemize}
\item $q < a_n$, or
\item $q > a_n$ and $p$ is not of the form $bq+1$ for some $b>0$, or
\item $q > a_n$ and $p=bq+1$ for some $b>0$, but $s_1>1$, or
\item $q > a_n$, $p=bq+1$ for some $b>0$, and $s_1=1$, but $a_2\neq a_1+1$,
\end{itemize}
then $T((a_1, s_1a_1), (a_2, s_2a_2), \dots, (a_n, s_na_n), (p, q))$ is not a torus knot. \qed
\end{theorem}

\begin{theorem}\label{Thm:2Twists_an_p_RelPrime}
Let $a_1, \dots, a_n$, $s_1, \dots, s_n$, and $p, q, k$ be integers satisfying the following hypotheses:
\begin{itemize}
\item $p$ and $q$ are relatively prime,
\item $1<a_1<\dots < a_n$ and $1 < q < a_n < p$,
\item each $s_i>0$ and $s_n\geq 2$,
\item $p$ and $a_n$ are relatively prime,
\item $k\geq 2$, $n\geq 2$.
\end{itemize}
In addition, suppose one of the following holds:
\begin{itemize}
\item $q \neq 1$,  
\item or $s_1 > 1$,
\item or $a_2 \neq a_1 +1$.
\end{itemize}
Then the T-link
$K = T((a_1, a_1s_1), \dots, (a_n, a_ns_n), (p, q + kp))$
is hyperbolic.
\end{theorem}

\begin{proof}
Because $\gcd(p,q)=1$, $K$ is a knot. By \refprop{2Twists_an_p_RelPrime}, $K$ is atoroidal, so not a satellite knot. 

The T-knot $K$ is equivalent to the T-knot
\[T((a_1, a_1s_1), \dots, (a_n, a_ns_n), (q + kp, p)) \]
by~\cite[Corollary~3]{BirmanKofman}.  The integer $q + kp$ does not have the form $bp + 1$ if and only if $q$ is different from 1. So under these conditions, $K$ is not a torus knot by \refthm{NonTorusKnots}.

Therefore, by Thurston's hyperbolization Theorem for knots~\cite{thurston:bulletin}, $K$ is hyperbolic.
\end{proof}

\begin{proposition}\label{Prop:SomeParamsAtoroidal}
Let $a_1, \dots, a_n$, $s_1, \dots, s_n$, and $p, q, k$ be integers satisfying the following hypotheses:
\begin{itemize}
\item $p$ and $q$ are relatively prime,
\item $1<a_1<\dots < a_n$, and $1 < q < a_n < p$,
\item each $s_i>0$ and both $s_n$ and $s_{n-1}$ are at least $2$.
\end{itemize}
Suppose also that one of the following holds:
\begin{itemize}
\item $q < a_{n-1}$ and $a_n$ and $a_{n-1}$ are relatively prime, or
\item $q > a_{n-1}$ and $a_n$ and $q$ are relatively prime.
\end{itemize}
Then the knot
$ K = T((a_1, a_1s_1), \dots, (a_n, a_ns_n), (p, q)) $
is atoroidal.
\end{proposition}

\begin{proof}
Suppose the exterior of $K$ in $S^3$ admits an essential torus $T$. 

By \refprop{BraidMove}, $K$ is equivalent to the knot given by the closure of the braid
\[ B = (\sigma_{a_n-1}\dots \sigma_{a_n-q+1})^{p-a_n} \cdot \tau \cdot
(\sigma_1\dots \sigma_{a_n-1})^{s_na_n+q}, 
\]
where $\tau$ is the concatination of braids $(a_1,a_1 s_1)\dots (a_{n-1},a_{n-1}s_{n-1})$. 

Since $B$ has at least two positive full twists on $a_n$ strands, it follows from \cite[Theorem~1.2(3)]{Ito} that $T$ does not intersect the braid axis $C$ of $B$. Thus there is an integer $d>0$ such that $K$ is a generalized $d$-cabling of the core $L$ of the solid torus bounded by $T$. Hence $d$ divides $a_n$.

Perform $(-1/s_{n})$-Dehn surgery along the braid axis $C$ to obtain the braid
\[
B' = (\sigma_{a_n-1}\dots \sigma_{a_n-q+1})^{p-a_n}\cdot \tau \cdot
(\sigma_1\dots \sigma_{a_n-1})^{q}. \]
Its closure gives $K'=T((a_1,a_1s_1),\dots,(a_{n-1},a_{n-1}s_{n-1}), (a_n,q))$.
The torus $T$ becomes a new essential torus $T'$ in the exterior of $K'$. 
The torus $T'$ bounds a solid torus with core $L'$, which is either trivial or knotted.

Suppose first the case that $q < a_{n-1}$. Then $q \leq a_{n-1} \leq a_{n-1}s_{n-1}+q$, so \refcor{TLinkBraidIndex} implies that $K'$ has braid index equal to $a_{n-1}$.
If $L'$ is the trivial knot, then $a_{n-1}$ is equal to $d$ by \reflem{WilliamsTrivialCase}. This implies that $d$ divides both $a_n$ and $a_{n-1}$, contradicting the assumption in this case that these are relatively prime.
Similarly, if $L'$ is knotted, then \refthm{Williams} implies that $a_{n-1}$ is a multiple of $d$, with the same contradiction.

Now suppose $q > a_{n-1}$. Then $K'$ has braid index $q$ by Franks and Williams, \refthm{FranksWilliams}. 
If $L'$ is trivial, then as above, \reflem{WilliamsTrivialCase} implies $q$ equals $d$, and therefore $d$ divides both $a_n$ and $q$, contradicting the hypothesis. Similarly, if $L'$ is knotted, \refthm{Williams} implies $q$ is a multiple of $d$, and again $d$ divides both $a_n$ and $q$, which is a contradiction. 
\end{proof}

\begin{theorem}\label{Thm:SomeParamsHyperbolic}
Let $a_1, \dots, a_n$, $s_1, \dots, s_n$, and $p, q, k$ be integers satisfying the following hypotheses:
\begin{itemize}
\item $p$ and $q$ are relatively prime,
\item $1<a_1<\dots <a_n$ and $1 < q < a_n < p$,
\item each $s_i>0$ and both $s_n$ and $s_{n-1}$ are at least $2$.
\end{itemize}
Suppose also that one of the following holds:
\begin{itemize}
\item $q < a_{n-1}$ and $a_n$ and $a_{n-1}$ are relatively prime, or
\item $q > a_{n-1}$ and $a_n$ and $q$ are relatively prime.
\end{itemize}
Then $K = T((a_1, a_1s_1), \dots, (a_n, a_ns_n), (p, q))$
is hyperbolic.
\end{theorem}

\begin{proof}
By \refprop{SomeParamsAtoroidal}, the knot $K$ is atoroidal.
By \refthm{NonTorusKnots}, using the fact that $q<a_n$, $K$ is anannular.

Therefore, $K$ is hyperbolic.
\end{proof}

\section{Satellite T-links obtained by Half-twists}\label{Sec:Satellite}

In this section we switch from discussions of hyperbolic links to satellite links. We find families of Lorenz links that are satellites using half-twists, rather than full-twists. Previous work by de~Paiva and Purcell found conditions that ensure a T-link is satellite, namely~\cite[Theorem~4.3]{depaiva-purcell2021satellites}. Lee has similar results for the case of twisted torus knots~\cite[Theorem~1]{Lee:Cable}. We extend these results. 

\begin{definition}\label{Def:AugmentedHalfTwistD}
Suppose $B$ is a diagram given as a closed braid; we consider the braid to have strands running vertically on the plane of projection.
A \emph{positive half-twist} on the strands from $a$ to $b$ is the braid
\[ \Delta_{a,b} = (\sigma_a \dots \sigma_b)(\sigma_a\dots \sigma_{b-1})\dots (\sigma_a). \]
This can be thought of as cutting the braid between the $a$-th and $b$-th strands, rotating in the anticlockwise direction by $180^\circ$, and gluing back. In braid theory literature, the positive half-twist on all strands is well known as the Garside fundamental braid.
A \emph{negative half-twist} is defined similarly, only the rotation is in the clockwise direction. See \reffig{HalfTwistBraid}.
\end{definition}

\begin{figure}
  \includegraphics{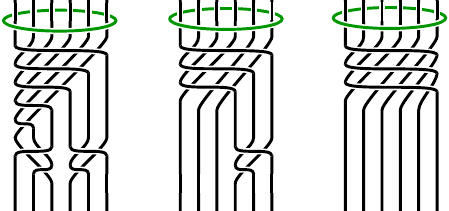}
  \caption{An example of half twists when $r=2, q=3, t=1$. Left: A positive half-twist $\Delta_{1,rq}$, a negative half-twist $\Delta_{1,(r-t)q}$ and a positive half-twist $\Delta_{(r-t)q+1,tq}$. The green circle indicates the braid axis. Middle: The negative half-twist cancels crossings above. Right: The additional positive half-twist gives the braid $(rq,tq)$.}
  \label{Fig:HalfTwistBraid}
\end{figure}

\begin{lemma}\label{Lem:BraidRqSq}
Let $r,q,s$ be positive integers, and suppose $s$ is not a multiple of $r$. The $(rq,sq)$-torus braid is obtained by the following procedure. Start with the trivial braid on $rq$ strands; let $J_{1,rq}$ be an unknot encircling all $rq$ strands. Let $t$ be an integer such that $0<t<r$ and $s=t+kr$ for some integer $k$. Insert a positive half-twist $\Delta_{1,rq}$, followed by a negative half-twist $\Delta_{1,(r-t)q}$ and a positive half-twist $\Delta_{(r-t)q+1,rq}$. 
Finally, perform $1/k$-Dehn filling on $J_{1,rq}$. The result is the $(rq,sq)$-torus braid. 
\end{lemma}

\begin{proof}
The process is illustrated in \reffig{HalfTwistBraid}. The positive half-twist $\Delta_{1,rq}$ yields a braid
\[(\sigma_1\sigma_2\dots\sigma_{rq-1})(\sigma_1\dots\sigma_{rq-2})\dots(\sigma_1),\]
encircled by $J_{1,rq}$. 
Perform the negative half-twist $\Delta_{1,(r-t)q}$. This concatenates the previous braid with
\[
(\sigma_{(r-t)q-1}^{-1}\dots\sigma_2^{-1}\sigma_1^{-1})(\sigma_{(r-t)q-1}^{-1}\dots\sigma_2^{-1})\dots(\sigma_{(r-t)q-1}^{-1}).\]
This braid cancels with the positive half-twist along the first $(r-t)q$ strands, as shown in \reffig{HalfTwistBraid}, middle. Finally, the positive half-twist $\Delta_{(r-t)q+1,rq}$ concatenates a positive half-twist along the last $tq$ strands, giving the braid
\[(\sigma_1\dots\sigma_{rq-1})^{tq}=(rq,tq),\]
still augmented by the unlink $J_{1,rq}$.

To obtain the braid $(rq,sq)$, perform $1/k$ Dehn filling on $J_{1,rq}$, removing that link component and inserting an additional $krq$ overstrands into the braid, for a total of $tq+krq=sq$ overstrands, giving the desired $(rq,sq)$-torus braid.
\end{proof}

\begin{lemma}\label{Lem:CylinderRqSq}
Let $r,q,s$ be positive integers, with $s$ not a multiple of $r$. Consider the torus braid $(rq,sq)$. At the top of the braid, consider $r$ disjoint discs arranged horizontally, each encircling $q$ strands of the braid, and similar discs at the bottom of the braid, as in \reffig{Cylinder}. The boundary of each disc at the top connects via a cylinder, embedded in the complement of the braid and enclosing $q$ strands, to the boundary of a disc at the bottom of the braid.

Moreover, the solid cylinders enclosed by these cylinders, containing $q$ strands each, forms the $(r,s)$-torus braid.
\end{lemma}

\begin{figure}
  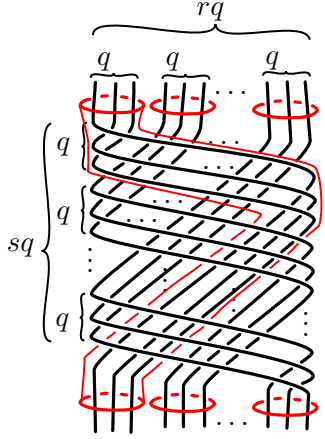
  \caption{The discs of \reflem{CylinderRqSq} are shown in red, for an example with $q=3$. One of the cylinders of that lemma is also sketched for this example. It follows the first $q$ strands. The other cylinders are similar.}
  \label{Fig:Cylinder}
\end{figure}

\begin{proof}
Let $t$ be an integer such that $0<t<r$ and $s=t+kr$ for some integer $k$. 
By \reflem{BraidRqSq}, the $(rq,sq)$ torus braid is formed from $k$ full twists on $rq$ strands, followed by a positive half-twist $\Delta_{1,rq}$, then a negative half-twist $-\Delta_{1,(r-t)q}$ and a positive half-twist $\Delta_{(r-t)q+1,rq}$. Each half-twist is on a multiple of $q$ strands.

Observe that the cylinders described above can be arranged to completely contain any half-twist on $q$ strands. For a half-twist on a multiple of $q$ strands, say $xq$ strands, $x$ disjoint cylinders enter the top of the half-twist, and then are half-twisted themselves, remaining disjoint, to exit the bottom of the half-twist. Thus the cylinders remain embedded as claimed when passing through half-twists. Finally, each full twist also preserves the cylinders, sending each through a full twist.

To see that the braid formed by the solid cylinders is as claimed, observe that the cylinders form $k$ full twists, followed by one positive half-twist on all strands. The $(r-t)$ left-most cylinders then pass through a negative half-twist, and the remaining $t$ right-most cylinders pass through a positive half-twist. As in \reflem{BraidRqSq}, this creates braid on $r$ strands, with $rk$ overstrands at the top coming from the full twists, followed by $t$ overstrands coming from the concatenation of half-twists. Thus this is an $(r,rk+t) = (r,s)$-torus braid.
\end{proof}

\begin{theorem}\label{Thm:MainHalfTwist}
Let $p,q$ be integers such that $1<q<p$, and let $(a_1,b_1)$, $\dots$, $(a_n, b_n)$ be pairs of integers such that $1<a_1<\dots< a_n\leq q $ and $b_i>0$ for $i=1, \dots, n$. Finally let $r_1, \dots, r_m$ and $s_1, \dots, s_m$ be integers such that $q< r_1q < \dots < r_mq < p$, and $s_i>0$ for $i=1, \dots, m$. Then the T-link
\[
K = T((a_1,b_1), \dots, (a_n, b_n), (r_1q,s_1q), \dots, (r_mq,s_mq),(p,q))
\]
is satellite with companion the T-link $T((r_1,s_1), \dots, (r_m,s_m+1))$. If this companion is a knot, then the pattern is given by the closure of the braid
\[
(a_1, b_1)\dots (a_n,b_n)(\sigma_{q-1}\dots \sigma_{1})^{p-r_m}\left(q,q\left(\sum_{i=1}^m r_is_i\right)+qr_m\right)
\]
\end{theorem}

\begin{proof}
As before, we think of the T-link as the closure of a braid on $p$ strands arranged vertically, the concatenation of braids $(a_1, b_1)$, $\dots$, $(a_n,b_n)$, $(r_1q,s_1q)$, $\dots$, $(r_mq,s_mq)$, $(p,q)$ in that order.

First apply \refprop{BraidMove} to change the closed braid of the T-link to a closed braid $B'$ on $r_mq$ strands. This isotopy fixes all of the $r_mq$ strands at the top left of the original braid; thus it does not affect any of the braids $(a_j,b_j)$ or $(r_iq, s_iq)$, for any $i$, $j$. In other words, $B'$ is the braid given by concatenating $(\sigma_{r_mq-1}\dots \sigma_{r_mq-q+1})^{p-r_m}$ with braids $(a_1, b_1)$, $\dots$, $(a_n,b_n)$, $(r_1q,s_1q)$, $\dots$, $(r_mq,s_mq)$, and finally the braid $(\sigma_1 \dots \sigma_{r_mq-1})^q$. 

By \reflem{CylinderRqSq}, there are $r_m$ disjoint embedded cylinders in the complement of the portion of the braid starting just above the braid $(a_1,b_1)$, and ending just below the braid $(\sigma_1\dots \sigma_{r_mq-1})^q$ at the bottom. These cylinders each enclose $q$ strands. They extend around the braid closure to give $r_m$ disjoint embedded cylinders running to the top of the braid, each enclosing $q$ strands, arranged right to left across the top of the braid.

The only portion of the braid that is not already enclosed in one of these cylinders is the braid $(\sigma_{r_mq-1}\dots \sigma_{r_mq-q+1})^{p-r_m}$ lying at the top. This is a braid whose left-most strand is the $(r_mq-q+1)$-th strand, and whose right-most strand is the $r_mq$-th strand. In other words, this is a braid on the right-most $q$ strands of the $r_mq$-strand braid. Thus the right-most cylinder, enclosing $q$ strands, can be extended to enclose this braid. Then all cylinders connect to form a closed embedded torus $\Sigma$, encircling $q$ strands of the braid. 

The torus $\Sigma$ bounds a solid torus containing $q$ strands, which we check has the claimed form of the companion in the theorem statement. This solid torus forms a braid on $r_m$ strands. By \reflem{CylinderRqSq}, each $(r_iq,s_iq)$-torus braid from the original T-link causes the solid cylinder to form a braid $(r_i,s_i)$. The braids $(a_j,b_j)$ and $(\sigma_{r_mq-1}\dots \sigma_{r_mq-q+1})^{p-r_m}$ lie completely inside the solid cylinder, so they do not affect the braid it forms. Finally, consider the braid $(\sigma_1\dots\sigma_{r_mq-1})^q$ at the bottom of $B'$. This is formed by $q$ strands running over all the $r_mq$ strands. When the collection of solid cylinders encounter this braid, the left-most solid cylinder encircles exactly these $q$ strands, and runs over all others to lie on the right-most side. Thus it forms a $(r_m, 1)$-torus braid. So the solid torus enclosing $q$ strands has the form of the closure of a braid $(r_1,s_1)\dots(r_m,s_m),(r_m,1)$. This is the T-link $T((r_1,s_1), \dots (r_m,s_m+1))$ as claimed. Since it forms a nontrivial knot in $S^3$, $\Sigma$ is an incompressible torus.

Finally we check the form of the pattern when the companion is a knot. Starting at the top-left of the braid $B'$, the torus $\Sigma$ encloses the braid $(a_1,b_1)\dots(a_n,b_n)$, which will form part of the braid describing the pattern. As $\Sigma$ follows the companion into each of the braids $(r_i,s_i)$, all the $q$ strands will make one full twist each time $\Sigma$ runs completely through an overstrand. There are $s_i$ of these, $i=1, \dots m-1$, plus $s_m+1$ for the $(r_m, s_m+1)$ braid that the companion runs over. These will occur in some order, with $\Sigma$ also enclosing the braid $(\sigma_{q-1}\dots \sigma_{1})^{p-r_m}$, coming from the top right of $B'$, at some point. Because full twists commute in the braid group, we may write the braid as $(a_1,b_1)\dots(a_n,b_n)(\sigma_{q-1}\dots \sigma_{1})^{p-r_m}\cdot \tau$ where $\tau$ is an appropriate number of full twists. To obtain the appropriate number of full twists, we need to consider the homological longitude of the companion. The pattern is the braid obtained when we apply a homeomorphism taking the solid torus bounded by the companion to an unknotted solid torus, with homological longitude mapped to a standard longitude of the unknot. The effect is to add $\sum_{i=1}^{m-1} (r_i-1)s_i + (r_m-1)(s_m+1)$ additional full twists, for a total of $\sum_{i=1}^m r_is_i +r_m$ full twists. Thus the pattern can be written as the braid
\[
(a_1, b_1)\dots (a_n,b_n)(\sigma_{q-1}\dots \sigma_{1})^{p-r_m}(q,q(\sum r_is_i)+qr_m) \qedhere
\]
\end{proof}


\bibliographystyle{amsplain}  

\bibliography{references}

\end{document}